\documentclass[11pt,a4paper]{article}
\usepackage[top=2.5cm, bottom=2.2cm, left=2.5cm, right=2.5cm]{geometry}
\usepackage[english]{babel}
\usepackage[utf8]{inputenc}
\usepackage[T1]{fontenc}
\usepackage{amsmath, amssymb, amsthm}
\usepackage{mathtools, enumerate,color}
\usepackage{tocloft}
\usepackage{array}
\usepackage{booktabs}
\usepackage[hidelinks]{hyperref}
\numberwithin{equation}{section}

\newtheorem{theorem}{Theorem}[section]
\newtheorem*{theorem*}{Theorem}
\newtheorem{proposition}[theorem]{Proposition}
\newtheorem{lemma}[theorem]{Lemma}
\newtheorem{corollary}[theorem]{Corollary}
\theoremstyle{definition}
\newtheorem{definition}[theorem]{Definition}

\newtheorem{remark}[theorem]{Remark}
\newtheorem*{remark*}{Remark}

\newcommand{\R}{\mathbb{R}}
\newcommand{\C}{\mathbb{C}}

\date{October 27,   2026}
\title{A Direct Approach to Hermite Interpolation}
\author{Simon Bossoney and Marc Troyanov\\ \small \'Ecole Polytechnique F\'ed\'erale de Lausanne
 \footnote{simon.bossoney@epfl.ch, \ marc.troyanov@epfl.ch}}

\begin{document}

\maketitle

\begin{abstract}
We introduce a family of polynomials satisfying the natural duality relations for Hermite interpolation, analogous to the classical Lagrange
interpolation polynomials. They yield an explicit closed formula for the Hermite interpolant with arbitrary multiplicities, without recourse to divided differences, recursive corrections, or auxiliary Bézout identities.
We also give a detailed account of Hermite's original approach to the problem, based on an integral formula, which he used both to derive the interpolating polynomial and to estimate the interpolation error for holomorphic data.

\medskip 


\smallskip

\noindent 
AMS Subject Class: {Primary 41A05; Secondary 13B25, 30E05.}
\end{abstract}
% Fin de l'abstract 

%____________
  % \tableofcontents
%____________
 
\section{Introduction} 
In the 1870s, the French mathematician Charles Hermite regularly corresponded with his German colleague Carl Wilhelm Borchardt, editor of the \textit{Journal für die reine und angewandte Mathematik}, founded by August Leopold Crelle in 1826. Several of these letters were published in the journal. One of them, entitled \emph{Sur la formule d'interpolation de Lagrange}, appeared in 1878~\cite{Hermite1878}. 
In this letter, Hermite considers the following interpolation problem, which we state here in a translation from French based on his own words:
\begin{quote}
\emph{Find a polynomial $F(x)$, of degree at most $n-1$, that satisfies the conditions}
\begin{align*}
 F(a) &= f(a), &F'(a)=f'(a), & \quad \ldots  , 
 &F^{(\alpha-1)}(a) =f^{(\alpha-1)}(a),\\
 F(b) &= f(b), &F'(b)=f'(b), & \quad \ldots, &
 F^{(\beta-1)}(b)=f^{(\beta-1)}(b),\\
 &\vdots &      \vdots    & \  \\
 F(l) &= f(l), &F'(l)=f'(l), & \quad \ldots, &
 F^{(\lambda-1)}(l)=f^{(\lambda-1)}(l),\\
\end{align*}
\emph{where $f$ is a given function and $\alpha+\beta+\cdots+\lambda=n$.
The problem is, as we see,  completely determined and leads to a generalization of Lagrange's interpolation formula.}
\end{quote}
When all multiplicities are equal to one, Hermite's problem indeed reduces to the
classical  \textit{Lagrange interpolation problem}:  Let $K$ be a commutative field and
let $x_1,\ldots,x_r\in K$ be pairwise distinct. Given $a_1,\ldots,a_r\in K$, one seeks the unique polynomial $F\in K[t]$ of degree at most $r-1$ satisfying
$$
 F(x_i)=a_i,\qquad i=1,\ldots,r.
$$
In that special case, the problem is solved using the   \textit{Lagrange interpolation polynomials}:
\begin{equation}\label{Lagrange}
\ell_i(t)=\prod_{j\neq i}\frac{t-x_j}{x_i-x_j},
\end{equation}
they are of degree at most  $r-1$ and satisfy the duality relation $\ell_i(x_j)=\delta_{ij}$. The classical Lagrange interpolation formula follows immediately:
\begin{equation}\label{LagrangeInterpolationFormula}
 F(t)=\sum_{i=1}^r a_i \ell_i(t).
\end{equation}
The purpose of this note is to show how the same idea extends, almost verbatim, to Hermite interpolation, replacing point evaluations by Taylor polynomials of prescribed order. 
This construction also yields a direct resolution of the partial fraction decomposition for rational functions and an explicit realization of the inverse isomorphism in the Chinese Remainder Theorem for polynomial algebras. 

%------------------------------------------------------------------
\section{The basic Hermite interpolation polynomials}
%------------------------------------------------------------------

We formulate Hermite's interpolation problem in a modern language. 
Let $K$ be an arbitrary field, let $x_1,\ldots,x_r\in K$ be pairwise
distinct, and let $m_1,\ldots,m_r$ be positive integers. For each $i$,
let $g_i\in K[t]$ be a polynomial of degree less than $m_i$. The collection
\begin{equation}\label{data}
 (x_i,m_i,g_i),\qquad i=1,\ldots,r,
\end{equation}
will be referred to as the \emph{interpolation data}, and the $x_i$'s as
the corresponding \emph{nodes}.

\smallskip 

\begin{definition} 
The \emph{Hermite interpolation problem} consists in finding a polynomial $f\in K[t]$ of degree less than $n=m_1+\cdots+m_r$ such that
\begin{equation}\label{hermcong}
    f\equiv g_i\pmod{(t-x_i)^{m_i}}, \qquad i=1,\ldots,r.
\end{equation}
The residue class of a polynomial modulo $(t-x_i)^{m_i}$ is called its
\emph{jet of order $m_i-1$} at $x_i$, so \eqref{hermcong} says that $f$
and $g_i$ have the same jet at $x_i$. The Hermite interpolation problem therefore consists of prescribing the jet of order $m_i-1$ of a polynomial $f$ of degree $<n$ at the node $x_i$ for $i = 1, \dots, r$. 
\end{definition} 

\begin{lemma}\label{lem.exun}
For any interpolation data, the Hermite interpolation problem has a unique
solution.
\end{lemma}
The unique solution is called the \emph{Hermite interpolant} associated to the given data.
\begin{proof}
Let $K_{< n}[t]$ be the space of polynomials of degree $<n$. The linear map
$$
 K_{< n}[t] \to \prod_{i=1}^r K[t]/(t-x_i)^{m_i}, \qquad 
 f\mapsto \big( f \bmod (t-x_1)^{m_1}, \dots, f \bmod (t-x_r)^{m_r}\big)
$$
is injective: indeed, if $f$ is in its kernel, then $f$ is divisible by each
$(t-x_i)^{m_i}$, hence by their product, which has degree $n$ since these factors are pairwise coprime; so $f=0$. Both spaces have dimension $n$,
so the map is bijective.
\end{proof}

\begin{remark}
If $K$ has characteristic zero, Taylor's formula shows that \eqref{hermcong} is equivalent to the classical formulation
\begin{equation}\label{hermderiv}
 f^{(k)}(x_i)=g_i^{(k)}(x_i),  \qquad 0\leq k<m_i,\quad 1\leq i\leq r,
\end{equation}
where $f^{(k)}$ denotes the $k$-th formal derivative of $f$. 

If $K = \R$ or $\C$ and $\displaystyle g_i =\sum_{k<m_i}\dfrac{g^{(k)}(x_i)}{k!}(t-x_i)^k$ is the Taylor polynomial at $x_i$ of a function $g$ with enough derivatives, we say that the interpolation data are \emph{derived from the function} $g$.
\end{remark}

\medskip

Our goal is to describe a direct construction of the Hermite interpolant,
based on elementary polynomial identities that satisfy duality relations
similar to those of the Lagrange basis. To this end, we introduce the following notations:
\begin{equation}\label{defomega}
\omega(t)=\prod_{j=1}^r (t-x_j)^{m_j},\qquad
\phi_i(t)=\prod_{j\ne i}(t-x_j)^{m_j},\qquad
q_i(t)=\frac{\phi_i(t)}{\phi_i(x_i)},
\end{equation}
and we define the \emph{basic Hermite interpolation polynomials} associated with the interpolation data \eqref{data} by
\begin{equation}\label{defhi}
  h_i(t) =  1 - \left( 1 - q_i(t) \right)^{m_i} 
  = 1-\left(1-\prod_{j\ne i}\left(\tfrac{t-x_j}{x_i- x_j}\right)^{m_j} \right)^{m_i}.
\end{equation}
We state the fundamental property satisfied by these polynomials in the following 
\begin{proposition}\label{prop:hermite-duality} 
For each $i,j\in\{1,\ldots,r\}$, the polynomials $h_i$ satisfy
\begin{equation}\label{dualityrel}
 h_i \equiv \delta_{ij} \quad \bmod{(t-x_j)^{m_j}}.
\end{equation}
\end{proposition}

\begin{proof}
From the definition of $q_i$, it is clear that
$$
(1-q_i)\equiv0\ \bmod{(t-x_i)}
\quad\text{and}\quad
(1-q_i)\equiv1\ \bmod{(t-x_j)^{m_j}} \quad(j\ne i).
$$
Raising these congruences to the power $m_i$ gives
$$
 (1-q_i)^{m_i}\equiv0\ \bmod{(t-x_i)^{m_i}}
  \quad\text{and}\quad
 (1-q_i)^{m_i}\equiv1\ \bmod{(t-x_j)^{m_j}}.
$$
Since $h_i=1-(1-q_i)^{m_i}$, we conclude that 
$$
 h_i\equiv1\ \bmod{(t-x_i)^{m_i}} \quad\text{and}\quad
 h_i\equiv0\ \bmod{(t-x_j)^{m_j}} \quad(j\ne i).
$$
\end{proof}

\medskip

Since a polynomial divisible by $(t-x_j)^{m_j}$  has all its derivatives of order less than $m_j$ vanishing at $x_j$, Proposition \ref{prop:hermite-duality}  implies the following:
\begin{corollary}\label{cor:hermite-derivatives}
For every $i,j\in\{1,\ldots,r\}$ we have: 
\begin{equation}\label{hiduality}
h_i(x_j)=\delta_{ij}, \quad\text{and}\quad h_i^{(k)}(x_j)=0 \ \text{ for } 0<k<m_j.
\end{equation}
\end{corollary}

\medskip

We also have the following identities:
\begin{corollary}\label{cor:idempotents}
The polynomials $h_i$ satisfy the following congruences modulo $\omega$:
\begin{equation}\label{eq:idempotent-relations}
  h_i^2\equiv h_i \quad (i=1,\ldots,r),
  \qquad
  h_ih_k\equiv 0 \quad  (\text{if } \  i\ne k), 
  \qquad
  \sum_{i=1}^r h_i\equiv 1.
\end{equation}
\end{corollary}

\begin{proof}
Since $x_1,\ldots,x_r$ are pairwise distinct, the polynomials $(t-x_1)^{m_1},\ldots,(t-x_r)^{m_r}$ are pairwise coprime. Therefore it suffices to verify each congruence modulo $(t-x_j)^{m_j}$ for every $j\in\{1,\ldots,r\}$.

\smallskip 

Fix $j$. By Proposition~\ref{prop:hermite-duality}, we have  $h_i\equiv\delta_{ij}\ \bmod{(t-x_j)^{m_j}}$  for every $i$.
Hence, modulo $(t-x_j)^{m_j}$,
$$
h_i^2\equiv\delta_{ij}^2=\delta_{ij}\equiv h_i,
\qquad
h_ih_k\equiv\delta_{ij}\delta_{kj}=0 \quad (i\ne k),
\qquad
\sum_{i=1}^r h_i\equiv\sum_{i=1}^r\delta_{ij}=1.
$$
Since this holds for every $j$ and the factors
$(t-x_j)^{m_j}$ are pairwise coprime, their product $\omega$
divides each of the polynomials
$$
h_i^2-h_i,\qquad h_ih_k\quad(i\ne k),\qquad
\sum_{i=1}^r h_i-1.
$$
Hence all the  congruences
\eqref{eq:idempotent-relations} hold modulo~$\omega$.
\\ 
\end{proof}

\medskip

\begin{remark}\label{re.prophi}
Proposition \ref{prop:hermite-duality}  and Corollary \ref{cor:idempotents} hold more generally: 
Let $\mathcal A$ be a commutative ring with unity, let $a_1,\ldots,a_r\in\mathcal A$,
let $m_1,\ldots,m_r\ge 1$ be integers, and let $q_1,\ldots,q_r\in\mathcal A$ satisfy
\[
 q_i\equiv 1 \pmod{a_i}
 \quad\text{and}\quad
 q_i\equiv 0 \pmod{a_j^{m_j}} \quad (j\neq i).
\]
Then $h_i=1-(1-q_i)^{m_i}$ satisfies $h_i\equiv\delta_{ij}\pmod{a_j^{m_j}}$ for all $i,j$.
If, moreover, the elements $a_1,\ldots,a_r$ are pairwise comaximal
(this is the case e.g. if $\mathcal A$ is a principal ideal domain and the $a_j$ are pairwise coprime),  then the congruences \eqref{eq:idempotent-relations} hold modulo $\alpha = \prod_{j=1}^r a_j^{m_j}$. 
\end{remark}

\medskip

%-----------------------------------------------------------------------
\section{The Hermite interpolation formula and its algebraic consequences}
%-----

\subsection{The Hermite interpolation formula}
We now formulate our main result: 
\begin{theorem}\label{thm:hermite-interpolation}
The unique solution $f\in K[t]$ to the Hermite interpolation problem is the  remainder modulo $\omega$ of the polynomial
\begin{equation}\label{solhermiteh}
 F(t) = \sum_{i=1}^{r} h_i(t)g_i(t).
\end{equation}
In other words, $f \in K[t]$ is the unique polynomial with $\deg f<n$, such that 
$$
f(t)\equiv F(t)  \ \bmod (\omega).
$$
\end{theorem}

Note that, when all multiplicities are equal to one, we have $h_i(t)=\ell_i(t)$ and  the formula above reduces to the classical Lagrange interpolation formula \eqref{LagrangeInterpolationFormula}.

\medskip 

\begin{proof}
Let $F(t)=\sum_{i=1}^{r}g_i(t)h_i(t)$; by Proposition~\ref{prop:hermite-duality}, we have
$$
F(t)\equiv g_j(t)\ \bmod{(t-x_j)^{m_j}}
$$
for each $j$. Now let $f$ be the remainder of $F$  modulo $\omega$; we still have
$$
f\equiv g_j\ \bmod{(t-x_j)^{m_j}}.
$$
Thus $f$ satisfies the Hermite interpolation conditions, and by construction $\deg f<n$.
Uniqueness has been in proved in Lemma \ref{lem.exun}. 
\end{proof}

\medskip 

To the best of our knowledge, the above construction of the Hermite interpolation polynomials $h_i$, and 
the resulting interpolation formula \eqref{solhermiteh},  have not appeared in the existing literature.

\subsection{The Chinese Remainder Theorem}
Theorem~\ref{thm:hermite-interpolation} and its proof also provide us with an explicit realization of the Chinese Remainder Theorem isomorphism in the case of polynomial algebras for split polynomials.  This Theorem states that if $x_1, \dots, x_r \in K$ are pairwise distinct, then  the canonical ring homomorphism
$$
K[t]/(\omega)\longrightarrow K[t]/(t-x_1)^{m_1}\times\cdots\times K[t]/(t-x_r)^{m_r}
$$
is an isomorphism, where $\omega$ is defined in \eqref{defomega}. The proof of Theorem~\ref{thm:hermite-interpolation} yields the explicit inverse
$$
\Bigl(g_1 \ \bmod{(t-x_1)^{m_1}},\ldots, g_r \ \bmod{(t-x_r)^{m_r}}\Bigr)
\longmapsto \sum_{i=1}^r h_i g_i \ \bmod{(\omega)}.
$$
Equivalently, every polynomial $F\in K[t]$ satisfies
$$
F\equiv \sum_{i=1}^r h_i g_i
\ \bmod{(\omega)},
$$
where $g_i$ denotes the remainder of $F$ modulo $(t-x_i)^{m_i}$.  Unlike the standard approach to the Chinese Remainder Theorem, the inverse is obtained here directly, without solving an auxiliary Bézout identity.

\smallskip

The relation between the Chinese Remainder Theorem and the classical Lagrange interpolation formula (i.e. the case all $m_i=1$) is well known. See e.g. P. Jolissaint~\cite{Jolissaint} for a recent discussion.

\subsection{Application to partial fraction decomposition}
The partial fraction decomposition problem, in the case where $\omega$ splits over $K$, boils down to finding polynomials $s_i\in K[t]$ with $\deg s_i<m_i$ such that
$$
\frac{1}{\omega(t)} = \sum_{i=1}^r\frac{s_i(t)}{(t-x_i)^{m_i}}.
$$
This is equivalent to finding polynomials $s_i$ satisfying the Bézout identity
$\sum_{i=1}^r s_i(t)\,\phi_i(t)=1$.
The polynomials $h_i$ defined in \eqref{defhi} provide the following explicit solution:  Since $h_i$ is  divisible by $\phi_i$ by Proposition~\ref{prop:hermite-duality}, we may define $s_i$ as the remainder of
$h_i/\phi_i$ upon division by $(t-x_i)^{m_i}$.
This polynomial satisfies
$$
 s_i\in K[t], \quad \deg(s_i) < m_i, \qquad 
 s_i(t) = \frac{h_i(t)}{\phi_i(t)}  \ \bmod (t-x_i)^{m_i}.
$$
By construction, we have 
$$
s_i\phi_i\equiv h_i\ \bmod{(t-x_i)^{m_i}},
$$
while for $j\neq i$ both  $s_i\phi_i$ and $h_i$ are congruent to $0$ modulo $(t-x_j)^{m_j}$. Hence
$$
s_i\phi_i\equiv h_i    \ \bmod{(\omega)}
$$
for every $i$. Corollary~\ref{cor:idempotents} then implies
\begin{equation}\label{consphi}
 \sum_{i=1}^r s_i(t)\,\phi_i(t)\equiv 1\ \bmod{(\omega)}.
\end{equation}
Moreover,
$$
\deg\!\left(\sum_{i=1}^r s_i\phi_i\right)\le n-1 < n = \deg(\omega),
$$
so the congruence \eqref{consphi} is in fact an equality. Dividing by $\omega(t)$, we obtain
$$
\frac{1}{\omega(t)}
=
\sum_{i=1}^r\frac{s_i(t)}{(t-x_i)^{m_i}}
=
\sum_{i=1}^r\sum_{k=1}^{m_i}\frac{a_{ik}}{(t-x_i)^k},
$$
where each polynomial $s_i$ has been expanded as $s_i(t)=\sum_{k=1}^{m_i}a_{ik}(t-x_i)^{m_i-k}$.

\section{Related constructions and historical remarks}
 
\subsection{Comparison with a construction by Stoer and Bulirsch} \label{sec.Stoer}

In characteristic zero, it is instructive to compare the present construction with the recursive construction described by Stoer and Bulirsch on \cite[\S 2.1.5]{Stoer}.  There, the Hermite interpolant is written as
\begin{equation}\label{solhermiteL}
f(t)=\sum_{i=1}^r\sum_{k=0}^{m_i-1} g_i^{(k)}(x_i)L_{i,k}(t),
\end{equation}
where $L_{i,k}  \in K[t]$ is a polynomial of degree $< n$  associated with $x_i$ and $0 \leq k < m_i$.  To define it, start from the auxiliary polynomials
$$
  q_{ik}(t) = \frac{(t-x_i)^k}{k!} \; q_i(t) 
   =  \frac{(t-x_i)^k}{k!}\prod_{j\neq i}\left(\tfrac{t-x_j}{x_i-x_j}\right)^{m_j};
$$
and set $L_{i,m_i-1}=q_{i,m_i-1}$,  then construct the remaining polynomials by descending recurrence using 
$$
L_{i,k}(t) =q_{ik}(t)-\sum_{s=k+1}^{m_i-1}q_{ik}^{(s)}(x_i)L_{i,s}(t),
$$
for $k= m_i-2,m_i-3,\dots, 0$. By construction,  every polynomial $L_{i,k} $ has at most degree $n-1$. By induction, one checks that
\begin{equation}\label{Likduality}
 L_{i,k}^{(s)}(x_j) = \delta_{ij}\,\delta_{ks},
\end{equation}
for every $i,j\in\{1,\ldots,r\}$, $0\le k<m_i$, and $0\le s<m_j$, from which \eqref{solhermiteL} follows immediately.

\medskip 

Note that the duality condition  \eqref{dualityrel} satisfied by the $h_i$ is simpler than condition \eqref{Likduality}
and that only $r$ polynomials are involved instead of $n$. 
The price to pay is that the duality relations for the
$h_i$ hold only modulo $(t-x_j)^{m_j}$; the polynomial
$F=\sum_{i=1}^r g_i h_i$ generally has degree at least $n$
and must be reduced modulo $\omega$.

\subsection{Hermite's original approach}  \label{sec:hermite-approach} 

Hermite's original approach deals with the complex field $K=\C$ and is based on the Cauchy integral formula. Assuming that the interpolation data are derived from a holomorphic function $g$ defined in  some neighborhood of the closure of a Jordan domain $\Omega\subset\mathbb C$ with rectifiable boundary that contains all the nodes, he proves the following remarkable  integral representation for the Hermite interpolant~$f$:
\begin{equation}\label{integralrepr}
 f(x)=g(x)+\frac{1}{2\pi i}\int_{\partial\Omega}\frac{g(z)\omega(x)}{(x-z)\omega(z)}\,dz,
\end{equation}
where $x\in \Omega$ and $\omega(x)=\prod_i(x-x_i)^{m_i}$; equivalently, 
$$
 f(x) = \sum_i  \mathrm{Res}_{x_i} \left(\frac{g(z)\omega(x)}{(x-z)\omega(z)}\right). 
$$
\medskip

Let us briefly explain Hermite's proof: We first show that $f(x)$ is a polynomial of degree $<n$, as follows.
Observe that $\omega(x)-\omega(z)$ is a polynomial of degree $n$ in the
variables $(x,z)$ that vanishes when $z=x$; therefore
$$
\omega(x) - \omega(z) = (x-z)Q(x,z),
$$
for a uniquely defined polynomial $Q(x,z)$ of degree $n-1$. Substituting into the integrand yields
$$
\frac{g(z)\omega(x)}{(x-z)\omega(z)} = \frac{g(z)}{x-z} + \frac{g(z)Q(x,z)}{\omega(z)}.
$$
Using that $g$ is holomorphic on a neighborhood of $\bar{\Omega}$, we have by Cauchy's integral formula,
$$
\frac{1}{2\pi i}\int_{\partial\Omega}\frac{g(z)}{x-z}\,dz =  -g(x),
$$
for every $x\in\Omega$. We thus have with \eqref{integralrepr}:
$$
  f(x) = \frac{1}{2\pi i}\int_{\partial\Omega}\frac{g(z)Q(x,z)}{\omega(z)}\,dz.
$$
This shows that $f$ is a polynomial in $x$ of degree at most $n-1$, since
$Q(x,z)$ is itself such a polynomial for every $z$.

\medskip

We now check that $f$ satisfies the Hermite interpolation conditions. It
will be convenient to write \eqref{integralrepr} as
\begin{equation}\label{integralrepr2}
 f(x) - g(x) = \frac{1}{2\pi i}\int_{\partial\Omega}\frac{g(z)\omega(x)}{(x-z)\omega(z)}\,dz.
\end{equation}
Observe that, at every node $x_i$, we have $\omega(x_i) = 0$, hence 
$f(x_i) = g(x_i)$. In fact we have 
$$
\left.\frac{d^k}{dx^k}\!\left(\frac{\omega(x)}{x-z}\right)\right|_{x=x_i}=0,
\qquad 0\le k<m_i,
$$
Thus $f$ and $g$ agree, along with their first $m_i-1$ derivatives, at every node $x_i$.
Combined with the degree bound above, this shows that $f$, as given by
\eqref{integralrepr}, is indeed the (unique) solution to the Hermite
interpolation problem.

\medskip

Hermite does not stop here: he also shows that, as the interpolation data grow (through additional nodes and/or higher multiplicities), the interpolating polynomial converges to the holomorphic function $g$; provided the nodes stay in a fixed region $U \Subset \Omega$.

\smallskip 

More precisely, the argument is the following: Let $U \subset \Omega$ be an open subset such that 
\begin{equation} \label{cond.geom}
  \mathrm{diam} (U) \leq \rho \;  \mathrm{dist}(U, \partial \Omega)
\end{equation}
for some positive $\rho < 1$. We assume also that  all interpolation nodes belong to $U$ and $\partial \Omega$ has finite length. 
Then for any $x\in U$ and $z \in \partial \Omega$ we have 
$$
 \max_i\frac{|x-x_i|}{|z-x_i|}\le\rho.
$$
Therefore, 
$$
\left|\frac{\omega(x)}{\omega(z)}\right| =\prod_i\left(\frac{|x-x_i|}{|z-x_i|}\right)^{m_i} \le\rho^n,
$$
where $n=\sum_i m_i$. Since $g$ is continuous on the compact set $\partial\Omega$ and $x\in U$ stays at a positive distance from $\partial\Omega$, there exists $M < \infty$ such that 
$$
  \left| \frac{g(z)}{x-z} \right| \leq M,
$$
for all $x\in U$ and $z\in\partial\Omega$.
Applying this to \eqref{integralrepr2} yields the following error estimate for any $x\in U$:
\begin{equation} \label{estimerror}
|f(x)-g(x)|\le  \frac{1}{2\pi}\int_{\partial\Omega} \left|\frac{g(z)\omega(x)}{(x-z)\omega(z)}\right|\, |dz|
\leq 
\frac{M}{2\pi}\   {\mathrm{length}(\partial\Omega)} \,\rho^n. 
\end{equation}
Since $\rho<1$, the right-hand side tends to $0$ as $n\to\infty$;
consequently, the interpolation error tends uniformly to zero on $U$.\footnote{Condition~\eqref{cond.geom} is sufficient for the estimate \eqref{estimerror} but not sharp. The argument only uses the weaker condition that there exists $\rho<1$ such that
$|x-x_i|\le\rho\,|z-x_i|$ for every $x\in U$, every $z\in\partial\Omega$
and every node $x_i$.}

\medskip

In the special case where $g$ is entire and the interpolation nodes remain in a fixed bounded region, the same argument, applied to sufficiently large contours, yields uniform convergence on every compact subset of $\C$; see, e.g., \cite[Theorem~1.19]{Calvi2005}.

\section{Concluding remarks}

\begin{enumerate}
\item After discussing the classical Lagrange interpolation formula in \cite[p.~50]{Walsh1960},
J.L. Walsh observes that ``an analogous formula exists'' for the Hermite interpolation problem,
but leaves its construction to the reader. Theorem~\ref{thm:hermite-interpolation} may be viewed
as one possible construction.

\item In the same reference, Walsh also states Hermite's integral representation for
interpolation with arbitrary multiplicities \cite[p.~50]{Walsh1960}, see also
\cite[\S 7]{Walsh1935}. His treatment, however, differs slightly from Hermite's
original argument and remains rather concise. Section~\ref{sec:hermite-approach}
above revisits Hermite's argument in detail, making explicit both the Cauchy-integral argument leading to the interpolation formula and the geometric condition underlying the convergence estimate

\item The degree of the polynomial $h_i$ is $d_i = m_i(n-m_i)$, which can be
significantly larger than $n$. From a computational viewpoint, it is thus useful to
reduce each $h_i$ modulo $\omega$ before forming the sum \eqref{solhermiteh}.  

\item Hermite's convergence argument relies on a geometric hypothesis on the location of the interpolation nodes relative to the boundary of the domain. Without such a restriction, uniform convergence might fail, as was observed by Runge's famous counterexample (1901) for polynomial interpolation at equally spaced nodes on some interval, see \cite{Epperson}.

\item Using the polynomials $h_i$, one produces a  direct construction of the spectral projectors of a square matrix. See \cite {BT2} for a detailed discussion.
\end{enumerate}

\smallskip

\emph{Acknowledgements.} The authors thank Marco Picasso for pointing out the reference~\cite{Stoer}.

%-----------------------------------------------------------------------


\begin{thebibliography}{9}

\bibitem{BT2}
S.~Bossoney and M.~Troyanov,
\textit{A Direct Polynomial Approach to Spectral Decomposition},  (2026),  \url{https://arxiv.org/abs/2607.20218}



\bibitem{Calvi2005}
J.-P. Calvi,
\textit{Lectures on Multivariate Polynomial Interpolation},
Lecture notes, Hanoi, 2005.
Available at:
\url{https://www.math.univ-toulouse.fr/~calvi/res_fichiers/MPI.pdf}

\bibitem{Epperson}
J. F. Epperson,
\textit{On the Runge Example},
Amer. Math. Monthly
\textbf{94} (1987), no.~4, 329--341.


\bibitem{Hermite1878}
C. Hermite,
\textit{Sur la formule d'interpolation de Lagrange.} Journal für die reine und angewandte Mathematik (1878)
Vol.84, pp. 70-79.
 
\bibitem{Jolissaint}
P. Jolissaint,
\textit{Interplay between the Chinese Remainder Theorem and the Lagrange Interpolation Formula},
Elemente der Mathematik,
\textbf{81} (2026), 22--25.

\bibitem{Stoer}
J. Stoer and R.~Z. Bulirsch,  \textit{Introduction to numerical analysis},
Texts in Applied Mathematics, 12, Springer, New York, 2002.

\bibitem{Walsh1935}
J. L. Walsh,
\textit{Approximation by polynomials in the complex domain}, 
Mémorial des sciences mathématiques, fascicule 73 (1935).

\bibitem{Walsh1960}
J. L. Walsh,
\textit{Interpolation and Approximation by Rational Functions in the Complex Domain},
American Mathematical Society Colloquium Publications,
Vol.~20, Various editions.
\end{thebibliography}
\end{document}